\documentclass[5p]{elsarticle}

\usepackage{lineno,hyperref}
\modulolinenumbers[5]
\usepackage{mathrsfs}
\usepackage{amsfonts}
\usepackage{booktabs}
\usepackage{amssymb}
\usepackage{bm}
\usepackage{multirow,bigstrut}
\usepackage{enumerate}
\usepackage{tikz}
\usepackage{graphicx}
\usepackage{stfloats,float}
\usepackage{amsmath}
\newtheorem{theorem}{Theorem}[section]
\newtheorem{proposition}{Proposition}[section]
\newtheorem{definition}{Definition}[section]
\newtheorem{remark}{Remark}[section]

\newtheorem{lemma}{Lemma}[section]
\newtheorem{example}{Example}[section]

\newenvironment{proof}{\noindent{\em \textbf{Proof.}}}{\quad \hfill$\Box$\vspace{2ex}}
\journal{Journal of \LaTeX\ Templates}

\begin{document}

\begin{frontmatter}

\title{Biquaternion Z Transform}

\author{Wenshan Bi}
\address{Department of Mathematics, Faculty of Science and Technology, University of Macau, Macao, China}

\author{Zhen-Feng Cai}
\address{School of Science, Hubei University of Technology, Wuhan, Hubei 430000, P.R.China}

\author{Kit Ian Kou\corref{mycorrespondingauthor}}
\address{Department of Mathematics, Faculty of Science and Technology, University of Macau, Macao, China}
\cortext[mycorrespondingauthor]{Corresponding author}
\ead{kikou@um.edu.mo}


\begin{abstract}
In this work, the biquaternion Z transformation method is proposed to solve a class of biquaternion recurrence relations. Biqueternion Z transform is an natural extension of the complex Z transform. In the design process, special norm presentation is employed to analyze the region of convergence of the biquaternion geometry sequence. In addition, some useful properties have been given. It is shown that the proposed properties is helpful to understand the biquaternion Z transform.
Finally, several examples have been given to illustrate the effectiveness of the proposed design method.
\end{abstract}


\end{frontmatter}

\linenumbers

\section{Introduction}
\label{intro}
During the past decades, Z transform has captured more and more attention from engineers due to its interesting properties and potential applications \cite{WH}-\cite{CTZapp5}.
The Z transform was first proposed  in 1947 to solve the difference equation \cite{WH}. This method  was applied to the transfer function problem of the impulse filter by Hurewicz \cite{Hurewicz}, and was called Z transforms by Ragazzini and Zadeh of Columbia University. The work of Ragazzini and Zadeh \cite{1952} includes an analysis of the input and output of the sampled-data systems. The Z transform of a sequence $f_n$ is defined as
$F(z)={\sum_{n=0}^\infty}f_nz^{-n}$,
which is a function of $z$ in complex domain. In general, both $f_n$ and $F(z)$ could be complex-valued. One of its applications is to establish the mathematical model of frequency-dependent finite-difference time domain method \cite{app}. Digital compensators can be used to improve the response of the sampled data system \cite{app1}-\cite{app3}. In particular, Jury and Schroeder \cite{app1} and Jury \cite{app2,app3} proposed to use the modified Z transform to examine the behavior between sampling transients and to design a digital compensator to suppress the "hidden" oscillations. In \cite{chirp}, Rabiner, Schafer and Rader together defined the chirp Z transform (CZT) based on the classical Z transform. The main idea is to use convolution to realize the fast Fourier transform of the discrete Fourier transform of any size, thus greatly reducing the computation time. In \cite{CTZapp1}, Leng and Song et al. applied CZT in frequency offset estimation of digital coherent orthogonal amplitude modulation, and solved the problem of uncertain delay. On the other hand, CZT plays an important role in the effective estimation of voltage flicker component in distribution system \cite{CTZapp2,CTZapp3}. Besides, in the channel prediction process of channel compensation in fading environment, CZT enables accurate detection even at low frequencies \cite{CTZapp4,CTZapp5}.

When working with sampled data the Z transform acts the role which is acted by the Laplace transform in continuous time systems. Specifically, in the continuous time domain model, the Laplace transform is usually used to transform the ordinary differential equation into the domain transfer function for theoretical analysis. When dealing with a system with discrete sampling time, the sequence model is transformed into a transfer function in the Z domain by using the Z transform.
It was shown in \cite{cai} that the quaternion differential equations with a continuous domain can be solved by Laplace transform method.

The Z transform of the biquaternion domain has rarely been discussed before. Generally speaking, the norm of biquaternion is complex-valued, which is difficult to analyze its convergence domain. In this paper, motivated by the fact that better results can be obtained when a special real-valued norm of biquaternion is exploited, a biquaternion Z transform method is presented.

The contributions of this work are summarized as follows.

\begin{enumerate}
	\item Although the theory of the Z transform has made rapid progress. However, this issue has not been discussed in the biquaternion domain. In this paper we give the suitable definition of biquaternion Z transform in Section \ref{d}.
	\item From the definition, we calculate the Z transform of some special functions, get many interesting properties. The basic properties and special function's Z transform are given in Table \ref{table:p1} and Table \ref{table:p2}.
	
	\item The difficulty process of solving a class of biquaternion recurrence relations is transformed into an algebraic quaternion problem. This method makes solving a class of biquaternion recurrence relations and the related initial values problem much convenient.
	Finally, the validity of the biquaternion Z transform is illustrated via several examples.
\end{enumerate}

The rest of this paper is organized as follows. The
definitions about biquaternion sequence, such as convergence of biquaternion sequence and biquaternion Z transform are given in section \ref{d}. In the following, we investigate its main properties and special biquaternion valued function's Z transform are given in section \ref{P}.  The method of using biquaternion z transform to solve biquaternion recurrence relations (\ref{difference}) is presented in section \ref{E}. Finally, we conclude this paper in section \ref{C}.

\section{Preliminary}\label{d}
\subsection{Biquaternions}
The  concept of biquaternion, a generalization of complex numbers, was first proposed by Hamilton in 1853. In general, we adopt the symbol $\mathbb{H(C)}$ to denote the set of biquaternion and $\mathbb{H(R)}$ to denote the set of real quaternion. For each biquaternion $\mathbf{q}$, it can be written as the following form
\begin{equation}\label{quaternion}
\mathbf{q}=q_0+q_1\mathbf{i}+q_2\mathbf{j}+q_3\mathbf{k},
\end{equation}
where $q_{n}\in\mathbb{C}, n=0,1,2,3$.
The basis elements $\{\mathbf{i},\mathbf{j},\mathbf{k}\}$ obey the Hamilton's multiplication rules
\begin{equation}\label{Hamilton's rules}
\begin{aligned}
\mathbf{i}^2&=\mathbf{j}^2=\mathbf{k}^2=\mathbf{i}\mathbf{j}\mathbf{k}=-1;\\
\mathbf{i}\mathbf{j}&=\mathbf{k},\mathbf{j}\mathbf{k}=\mathbf{i},\mathbf{k}\mathbf{i}=\mathbf{j};\\
\mathbf{j}\mathbf{i}&=-\mathbf{k},\mathbf{k}\mathbf{j}=-\mathbf{i},\mathbf{i}\mathbf{k}=-\mathbf{j}.\\
\end{aligned}
\end{equation}

For any biquaternion $\mathbf{p}, \mathbf{q} \in \mathbb{H(C)}$, we write $\mathbf{p}=p_0+p_1\mathbf{i}+p_2\mathbf{j}+p_3\mathbf{k}$, $\mathbf{q}=q_0+q_1\mathbf{i}+q_2\mathbf{j}+q_3\mathbf{k}$, where $p_{n}, q_{n}\in\mathbb{C}, n=0,1,2,3$. Let $\mathbf{p}=p_0+\underline{\mathbf p}$, we use $p_0:=Sc(\mathbf{p})$ be the scalar part of $\mathbf{p}$,  $\underline{\mathbf p}=:Vec(\mathbf{p})$ be the vector part of $\mathbf{p}$, and $\overline{\mathbf p}=:p_0-\underline{\mathbf p}$ be the quaternion conjugate of $\mathbf{p}$.
Define $(\underline{\mathbf p},\ \underline{\mathbf q}):=p_1q_1+p_2q_2+p_3q_3$ and
$\mathbf p \mathbf q:=p_0q_0- (\underline{\mathbf p},\ \underline{\mathbf q})+p_0\underline{\mathbf q}+q_0\underline{\mathbf p}+\underline{\mathbf p}\times \underline{\mathbf q}$.

In \cite{gurlebeck1997quaternionic}, G\"{u}rlebeck give the complex-valued norm in $\mathbb{H(C)}$.
Let $\mathbf x\in \mathbb{H(C)}$, $\mathbf{a}, \mathbf{b}\in \mathbb{H(R)}$, $a_{n}, b_{n}\in\mathbb{R}(n=0,1,2,3)$, $\mathbf x=\mathbf{a}+I\mathbf{b}$, where $\mathbf{a}=a_0+a_1\mathbf{i}+a_2\mathbf{j}+a_3\mathbf{k}=a_0+\underline{\mathbf a}$, $\mathbf{b}=b_0+b_1\mathbf{i}+b_2\mathbf{j}+b_3\mathbf{k}=b_0+\underline{\mathbf b}$, and $I$ be the complex unit. Then we adopt the symbol $\|\mathbf x\|_{\mathbb{C}}$ to denote the complex-valued norm of  the element $\mathbf x$, which has the following form
\begin{equation}\label{complex magnitude}
\|\mathbf x\|_{\mathbb{C}}^2:=\mathbf x\overline{\mathbf x}=|\mathbf{a}|^2-|\mathbf{b}|^2+2I[a_0b_0+(\underline{\mathbf a},\ \underline{\mathbf b})],
\end{equation}
where
$|\mathbf{a}|^2:=\mathbf a\overline{\mathbf a}$,  $|\mathbf{b}|^2:=\mathbf b\overline{\mathbf b}$, $\overline{\mathbf a}=:a_0-\underline{\mathbf a}$, $\overline{\mathbf b}=:b_0-\underline{\mathbf b}$. In general, we use ${\mathbf x}^{-1}:=\frac{\overline{\mathbf x}}{\|\mathbf x\|^2_{\mathbb{C}}}$ to denote the inverse of $\mathbf x$, where $\|\mathbf x\|_{\mathbb{C}}\neq0$.

Besides, G\"{u}rlebeck give the real-valued norm in $\mathbb{H(C)}$. Let ${\mathbf x}\in \mathbb{H(C)}$,  the real-valued norm $\|{\mathbf x}\|$ has the following form
\begin{equation}\label{norm}
\|\mathbf x\|^4=|\|\mathbf x\|^2_{\mathbb{C}}|^2.
\end{equation}
It's worth noting that the real-valued norm satisfies $\|\mathbf x \mathbf y\|=\|\mathbf x\|\|\mathbf y\|$. In this paper, we use $\|\cdot\|$ to represent real-valued norm of biquaternion as above defined unless we give a special explanation.

According to \cite{biq1,biq2}, the overall description of high-dimensional data by using biquaternion can effectively represent its inherent properties.
In this paper, in order to better distinguish domains, we use plain fonts to represent real or complex numbers, and use bold to represent real quaternions or biquaternions, unless we give a special explanation.
In addition, we use $\mathbf{i},\mathbf{j},\mathbf{k}$ to represent the imaginary unit of a quaternion number, noting that they are commutative with the complex unit $I$. For more details, please refer to \cite{Sangwine, gurlebeck1997quaternionic}.
\subsection{Biquaternion geometry sequence}
Before disscussing the definition of the Z transform in the $\mathbb{H(C)}$ domain, we first introduce the  convergence of biquaternion sequence.

An infinite sequence
\begin{equation}\label{sequence }
\mathbf{f}_{0}, \mathbf{f}_{1}, \cdots, \mathbf{f}_{n},  \cdots
\end{equation}
of biquaternion  has a limit $\mathbf{f}$, $\mathbf{f}\in \mathbb{H(C)}$, if, for each positive number $\varepsilon$, there exists a
positive integer $n_0$ such that
\begin{equation}
\| \mathbf{f}_{n}-\mathbf{f} \|<\varepsilon,
\end{equation}
whenever $n>n_0$. The sequence (\ref{sequence }) can have at most one limit. That is, a limit $\mathbf{f}$ is unique if it exists. When that limit exists, the sequence is said to converge to $\mathbf{f}$ ; and we write
\begin{equation}
\lim_{n\rightarrow\infty} \mathbf{f}_{n}=\mathbf{f}.
\end{equation}

Let $\mathbf{f}_{n}\in \mathbb{H(C)}$, $(n=1,2,\ldots)$. If the sequence
\begin{equation}\label{partial}
\mathbf{S}_N={\sum_{n=0}^N}\mathbf{f}_{n}=\mathbf{f}_{1}+\mathbf{f}_{2}+\cdots+\mathbf{f}_{N}
\end{equation}
of patital sums convergence to $\mathbf{S}$, $\mathbf{S}\in \mathbb{H(C)}$. Then we say the infinite sequence
\begin{equation}\label{series}
{\sum_{n=0}^\infty}\mathbf{f}_{n}=\mathbf{f}_{1}+\mathbf{f}_{2}+\cdots+\mathbf{f}_{n}+\cdots
\end{equation}
convergence to the sum $\mathbf{S}$, and we write
\begin{equation}\label{infty}
{\sum_{n=0}^\infty}\mathbf{f}_{n}=\mathbf{S}.
\end{equation}
When a sequence does not convergence, we say that it diverges. By direct calculate, we have the following proposition.

\begin{proposition}\label{theorem}
	Suppose $\mathbf{f}=\{\mathbf{f}_{n}\}_{n=0}^{\infty}$, $\mathbf{f}_{n}, \mathbf{S}\in \mathbb{H(C)}$, and $\mathbf{f}_n=m_n+\mathbf{i}w_n+\mathbf{j}x_n+\mathbf{k}y_n$, $\mathbf{S}=M+\mathbf{i}W+\mathbf{j}X+\mathbf{k}Y$, then
	\begin{equation}\label{1}
	{\sum_{n=0}^\infty}\mathbf{f}_{n}=\mathbf{S}.
	\end{equation}
	If and only if
	\begin{equation}\label{2}
	{\sum_{n=0}^\infty}m_{n}=M,
	{\sum_{n=0}^\infty}w_{n}=W,
	{\sum_{n=0}^\infty}x_{n}=X,
	\ {\sum_{n=0}^\infty}y_{n}=Y.
	\end{equation}
\end{proposition}
\subsection{Biquaternion Z transform}
In this section, we define the Z transform in the $\mathbb{H(C)}$ domain. Firstly, we introduce the biquaternion geometric sequence as follows.
\begin{lemma}\label{theorem1}
	Let $\mathbf{y}\in \mathbb{H(C)}$, we have
	\begin{equation}\label{e1}
	{\sum_{n=0}^\infty}\mathbf{y}^{n}=(1-\mathbf{y})^{-1},  \quad \|\mathbf{y}\|<1.
	\end{equation}
\end{lemma}

\begin{proof}
	Firstly, let $\mathbf{S}_N\in \mathbb{H(C)}$.  We use $\mathbf{S}_N(\mathbf{y})$ to denote the partial sums of ${\sum_{n=0}^\infty}\mathbf{y}^{n}$, that is
	\begin{equation}\label{8}
	\mathbf{S}_N(\mathbf{y})={\sum_{n=0}^{N-1}}\mathbf{y}^{n}=1+\mathbf{y}+\mathbf{y}^2+\cdots+\mathbf{y}^{N-1}.
	\end{equation}
	Then we multiply  Eq. (\ref{8}) by $\mathbf{y}$ on the left-hand side,
	\begin{equation}\label{9}
	\mathbf{y}\mathbf{S}_N(\mathbf{y})=\mathbf{y}{\sum_{n=0}^{N-1}}\mathbf{y}^{n}=\mathbf{y}+\mathbf{y}^2+\cdots+\mathbf{y}^{N}.
	\end{equation}
	Using Eq. (\ref{8})-(\ref{9})
	\begin{equation}\label{8-9}
	(1-\mathbf{y})\mathbf{S}_N(\mathbf{y})=1-\mathbf{y}^{N}.
	\end{equation}
	Thus we know
	\begin{equation}\label{10}
	\mathbf{S}_N(\mathbf{y})=(1-\mathbf{y})^{-1}(1-\mathbf{y}^{N}).
	\end{equation}
	When $\|\mathbf{y}\|<1$, $(1-\mathbf{y})\neq0$, $(1-\mathbf{y})^{-1}$ is well defined. Define
	\begin{equation}\label{11}
	\mathbf{S}(\mathbf{y}):=(1-\mathbf{y})^{-1},
	\end{equation}
	using Eq. (\ref{10})-(\ref{11})
	\begin{equation}\label{12}
	\begin{aligned}
	\mathbf{\rho}_N(\mathbf{y}):=&\mathbf{S}(\mathbf{y})-\mathbf{S}_N(\mathbf{y})\\
	=&(1-\mathbf{y})^{-1}-(1-\mathbf{y})^{-1}(1-\mathbf{y}^{N})\\
	=&(1-\mathbf{y})^{-1}\mathbf{y}^{N}.
	\end{aligned}
	\end{equation}
	Thus
	\begin{equation}\label{13}
	\|\mathbf{\rho}_N(\mathbf{y})\|=\|(1-\mathbf{y})^{-1}\mathbf{y}^{N}\|=\|(1-\mathbf{y})^{-1}\|\|\mathbf{y}\|^{N},
	\end{equation}
	and it is clear from this that the remainders $\rho_N(\mathbf{y})$ tends to zero when $\|\mathbf{y}\|<1$ but not when $\|\mathbf{y}\|\geq1$. Therefore, the proof of equation (\ref{e1}) is completed.
\end{proof}

As an immediate consequence of Proposition \ref{theorem} and Lemma \ref{theorem1}, we present the definition of biquaternion Z transform in the following.

\begin{definition}(Biquaternion Z transform)
	Let $\mathbf{f}=\{\mathbf{f}_n\}_{n=0}^\infty$, $\mathbf{f}_n, \mathbf{x} \in \mathbb{H(C)}$.
	The Z transform $\mathcal{X}[\mathbf{f}](\mathbf{x})$ of $\mathbf{f}$ is defined by
	\begin{equation}
	\mathcal{X}[\mathbf{f}](\mathbf{x}):=\mathbf{f}_{0}+\mathbf{f}_{1}{\mathbf{x}}^{-1}+\cdots+\mathbf{f}_{n}{\mathbf{x}}^{-n}+\cdots=
	{\sum_{n=0}^\infty}\mathbf{f}_{n}{\mathbf{x}}^{-n}.
	\end{equation}
	The region of convergence (ROC) is  the set of all $\mathbf{x}$ values that make the defined sequence converge, that is $\|\mathbf{x}\|>\sigma_f$. Among them, $\sigma_f~(0\leq{\|\sigma_f\|}<\infty)$ is defined as the radius of convergence of the Z transform of $\mathbf{f}$.
\end{definition}

\begin{remark}\label{de}
	\begin{enumerate}[(i)]
		\item
		If $\mathbf{f}_n=1$ for all $n\geq0$, the Z transform $\mathcal{X}[\mathbf{f}](\mathbf{x})$  of sequence $\{\mathbf{f}_n\}^\infty_{n=0}$ in $\mathbb{H(C)}$ is defined by
		\begin{equation}\label{e2}
		\mathcal{X}[\mathbf{f}](\mathbf{x}):={\sum_{n=0}^\infty}\mathbf{x}^{-n}=(1-\mathbf{x}^{-1})^{-1},
		\end{equation}
		which is convergent for all $\mathbf{x}$ such that $\|\mathbf{x}\|>1$. More precise, the ROC of $\mathcal{X}[\mathbf{f}](\mathbf{x})$ is $\|\mathbf{x}\|>1$.
		\item
		When $\mathbf{f}_n\in \mathbb{H(C)}(n\geq0)$, ${x}\in \mathbb{C}$. Let $\mathbf{f}_n=\mathbf{p}^n$, where $\mathbf{p}$ is a nonzero biquaternion number. The Z transform $\mathcal{X}[\mathbf{f}]({x})$ of sequence $\{\mathbf{f}_n\}^\infty_{n=0}$ is defined by
		\begin{equation}\label{definition}
        \begin{aligned}
		\mathcal{X}[\mathbf{f}]({x}):=&{\sum_{n=0}^\infty}\mathbf{f}_{n}{x}^{-n}=(1-\mathbf{p} {x}^{-1})^{-1},\\ &|{x}|>\|\mathbf{p}\|,
        \end{aligned}
		\end{equation}
		$\mathcal{X}[\mathbf{f}]({x})$ is convergence for all ${x}$ such that $|{x}|>\|\mathbf{p}\|$. More precise, the ROC of $\mathcal{X}[\mathbf{f}]({x})$ is $|{x}|>\|\mathbf{p}\|$.
		Define $\sigma_f$ to be the radius of convergence of the Z transform of $\mathbf{f}_n$. Here $\sigma_f:=\|\mathbf{p}\|$.
		\item
		When ${f}_n\in \mathbb{C}(n\geq0)$, $\mathbf{x}\in \mathbb{H(C)}$. Let ${f}_n={p}^n$, where ${p}$ is a nonzero complex number. The Z transform $\mathcal{X}[{f}](\mathbf{x})$ of sequence $\{{f}_n\}^\infty_{n=0}$ is defined by
		\begin{equation}\label{definition}
        \begin{aligned}
		\mathcal{X}[{f}](\mathbf{x}):=&{\sum_{n=0}^\infty}{f}_{n}\mathbf{x}^{-n}=(1-{p} \mathbf{x}^{-1})^{-1},\\ &\|\mathbf{x}\|>|{p}|,
		\end{aligned}
         \end{equation}
		$\mathcal{X}[{f}](\mathbf{x})$ is convergence for all $\mathbf{x}$ such that $\|\mathbf{x}\|>|{p}|$. More precise, the ROC of $\mathcal{X}[{f}](\mathbf{x})$ is $\|\mathbf{x}\|>|{p}|$.
		\item
		When ${x}, p\in \mathbb{C}$, and $p\neq0$. The Z transform $\mathcal{X}[{f}]({x})$ of the sequence ${f}_n=p^n~(n\geq0)$ is defined by
		\begin{equation}\label{classicalZ}
        \begin{aligned}
		\mathcal{X}[{f}]({x}):=&{\sum_{n=0}^\infty}{f}_{n}{x}^{-n}=(1-{p}{x}^{-1})^{-1},\\
& |{x}|>|{p}|,
        \end{aligned}
		\end{equation}
		where $|x|$ denote the absolute value of $x$.
	\end{enumerate}
\end{remark}

\section{Properties}\label{P}
Similar to complex Z transform, the biquaternion Z transform has many useful properties. First of all, the following theorem gives some calculation rules for Z transformation.
For the sake of simplicity, we have introduced some symbols for numerical sequence operations.
Here, $\sigma_f$ represents the radius of convergence of the Z transform of $\mathbf{f}_n$, which indicates that the sequence is convergent for $\|\mathbf{x}\|>\sigma_f$ and divergent for $\|\mathbf{x}\|<\sigma_f$.

\begin{theorem}\label{th1}
	Let $\mathbf{f}=\{\mathbf{f}_n\}_{n=0}^\infty$, $\mathbf{g}=\{\mathbf{g}_n\}_{n=0}^\infty$, $ \mathbf{f}_n, \mathbf{g}_n\in \mathbb{H(C)}$. $\mathcal{X}[\mathbf{f}]$ and $\mathcal{X}[\mathbf{g}]$ denote the biquaternion Z transform of $\mathbf{f} $ and $\mathbf{g}$, respectively.
	\begin{enumerate}[(I)]
		\item
		For all constants $\mathbf{c}_1, \mathbf{c}_2 \in \mathbb{H(C)}$, we have
		\begin{equation}\label{linear1}
        \begin{aligned}
		\mathcal{X}[\mathbf{c}_1\mathbf{f}+\mathbf{c}_2\mathbf{g}](\mathbf{x})
		=&\mathbf{c}_1\mathcal{X}[\mathbf{f}](\mathbf{x})+\mathbf{c}_2\mathcal{X}[\mathbf{g}](\mathbf{x}),\\ &\|\mathbf{x}\|>\max(\sigma_f,\sigma_g).
        \end{aligned}
		\end{equation}
		\begin{equation}\label{linear2}
        \begin{aligned}
		\mathcal{X}[\mathbf{f}\mathbf{c}_1+\mathbf{g}\mathbf{c}_2](\mathbf{x})
		=&\mathcal{X}[\mathbf{f}](\mathbf{x})\mathbf{c}_1+\mathcal{X}[\mathbf{g}](\mathbf{x})\mathbf{c}_2,\\ &\|\mathbf{x}\|>\max(\sigma_f,\sigma_g).
		\end{aligned}
        \end{equation}
		\begin{equation}\label{linear3}
        \begin{aligned}
		\mathcal{X}[\mathbf{c}_1\mathbf{f}+\mathbf{g}\mathbf{c}_2](\mathbf{x})
		=&\mathbf{c}_1\mathcal{X}[\mathbf{f}](\mathbf{x})+\mathcal{X}[\mathbf{g}]\mathbf{c}_2(\mathbf{x}),\\ &\|\mathbf{x}\|>\max(\sigma_f,\sigma_g).
        \end{aligned}
		\end{equation}
		\item  Let $\mathbf{q}$ be a nonzero biquaternion number and $\mathbf{g}_n=\mathbf{f}_n\mathbf{q}^n$, $n=0,1,2,\cdots$. If $\mathbf{q}\mathbf{x}=\mathbf{x}\mathbf{q}$, then
		\begin{equation}
		\mathcal{X}[\mathbf{g}](\mathbf{x})=\mathcal{X}[\mathbf{f}](\mathbf{q}^{-1}\mathbf{x}), \quad  \|\mathbf{x}\|>{\sigma_f}{\|\mathbf{q}\|}.
		\end{equation}
		\item Let $k>0$ be a fixed integer and $\mathbf{g}_n=\mathbf{f}_{n+k}$, $n=0,1,2,\cdots$, then
		\begin{equation}\label{th3}
		\begin{aligned}
		\mathcal{X}[\mathbf{g}](\mathbf{x})=&( \mathcal{X}[\mathbf{f}](\mathbf{x})-\mathbf{f}_0-\mathbf{f}_1\mathbf{x}^{-1}\\
&-\mathbf{f}_2\mathbf{x}^{-2}-\cdots-\mathbf{f}_{k-1}\mathbf{x}^{-(k-1)})\mathbf{x}^{k}\\ =&\mathcal{X}[\mathbf{f}](\mathbf{x})\mathbf{x}^{k}-\mathbf{f}_0\mathbf{x}^{k}\\
&-\mathbf{f}_1\mathbf{x}^{k-1}-\cdots-\mathbf{f}_{k-1}\mathbf{x},\\
& \|\mathbf{x}\|>\sigma_f.
		\end{aligned}
		\end{equation}
		\item Conversely, if $k$ is a positive integer and $\mathbf{g}_n=\mathbf{f}_{n-k}$ for $n\geq k$ and $\mathbf{g}_n=0$ for $n<k$, then $\mathcal{X}[\mathbf{g}](\mathbf{x})=\mathcal{X}[\mathbf{f}](\mathbf{x})\mathbf{x}^{-k}$. 
		\item Let $x$ be a complex-valued number and $\mathbf{g}_n=n\mathbf{f}_n$,$n=0,1,2,\cdots$,then
		\begin{equation}\label{33}
		\mathcal{X}[\mathbf{g}](x)=-x\frac{d}{dx}[\mathcal{X}[\mathbf{f}](x)], \quad |x|>\sigma_f.
		\end{equation}
	\end{enumerate}
\end{theorem}
\begin{proof}
	For (I), we only give the proof of Eq.(\ref{linear1}). Eq.(\ref{linear2}) and Eq.(\ref{linear3}) can be derived by the similar manner. Since $\mathcal{X}[\mathbf{f}](\mathbf{x})={\sum_{n=0}^\infty}\mathbf{f}_{n}\mathbf{x}^{-n}$ and $\mathcal{X}[\mathbf{g}](\mathbf{x})={\sum_{n=0}^\infty}\mathbf{g}_{n}\mathbf{x}^{-n}$, we have
	\begin{equation}
	\begin{aligned}
	\mathcal{X}[\mathbf{c}_1\mathbf{f}+\mathbf{c}_2\mathbf{g}](\mathbf{x})
	=&{\sum_{n=0}^\infty}\mathbf{c}_1\mathbf{f}_{n}\mathbf{x}^{-n}+{\sum_{n=0}^\infty}\mathbf{c}_2\mathbf{g}_{n}\mathbf{x}^{-n}\\
	=&\mathbf{c}_1{\sum_{n=0}^\infty}\mathbf{f}_{n}\mathbf{x}^{-n}+\mathbf{c}_2{\sum_{n=0}^\infty}\mathbf{g}_{n}\mathbf{x}^{-n}\\
	=&\mathbf{c}_1\mathcal{X}[\mathbf{f}](\mathbf{x})+\mathbf{c}_2\mathcal{X}[\mathbf{g}](\mathbf{x}),\\ &\|\mathbf{x}\|>\max(\sigma_f,\sigma_g).
	\end{aligned}
	\end{equation}
	For (II) we find
	\begin{equation}\label{pr2}
	\begin{aligned} \mathcal{X}[\mathbf{g}](\mathbf{x})=&{\sum_{n=0}^\infty}\mathbf{g}_{n}\mathbf{x}^{-n}\\
    =&{\sum_{n=0}^\infty}\mathbf{f}_n\mathbf{q}^n\mathbf{x}^{-n}\\
	=&{\sum_{n=0}^\infty}\mathbf{f}_n(\mathbf{q}^{-1}\mathbf{x})^{-n}\\
    =&\mathcal{X}[\mathbf{f}](\mathbf{q}^{-1}\mathbf{x}).
	\end{aligned}
    \end{equation}
	Since $\|\mathbf{q}^{-1}\mathbf{x}\|>{\sigma_f}$, therefore $\|\mathbf{x}\|>{\sigma_f}{\|\mathbf{q}\|}$.\\
	For (III),
	\begin{equation}\label{pr3}
	\begin{aligned}
	\mathcal{X}[\mathbf{g}](\mathbf{x})
    =&{\sum_{n=0}^\infty}\mathbf{g}_{n}\mathbf{x}^{-n}\\
    =&{\sum_{n=0}^\infty}\mathbf{f}_{n+k}\mathbf{x}^{-n}\\
	=&{\sum_{n=0}^\infty}\mathbf{f}_{n+k}\mathbf{x}^{-(n+k)}\mathbf{x}^{k}\\
	=&( \mathcal{X}[\mathbf{f}](\mathbf{x})-\mathbf{f}_0-\mathbf{f}_1\mathbf{x}^{-1}\\
&-\mathbf{f}_2\mathbf{x}^{-2}-\cdots-\mathbf{f}_{k-1}\mathbf{x}^{-(k-1)})\mathbf{x}^{k}\\
	=&\mathcal{X}[\mathbf{f}](\mathbf{x})\mathbf{x}^{k}-{\sum_{n=0}^{k-1}}\mathbf{f}_n\mathbf{x}^{k-n},\quad \|\mathbf{x}\|>\sigma_f.
	\end{aligned}
	\end{equation}
	For (IV),
	\begin{equation}\label{pr4}
	\begin{aligned}
	\mathcal{X}[\mathbf{g}](\mathbf{x})=&{\sum_{n=0}^\infty}\mathbf{g}_{n}\mathbf{x}^{-n}\\
=&{\sum_{n=0}^\infty}\mathbf{f}_{n-k}\mathbf{x}^{-n}\\
	=&{\sum_{n=0}^\infty}\mathbf{f}_{n-k}\mathbf{x}^{-(n-k)}\mathbf{x}^{-k}\\
	=&{\sum_{n=k}^\infty}\mathbf{f}_{n-k}\mathbf{x}^{-(n-k)} \mathbf{x}^{-k}\\
=&\mathcal{X}[\mathbf{f}](\mathbf{x})\mathbf{x}^{-k}
	,\quad \|\mathbf{x}\|>\sigma_f.
	\end{aligned}
	\end{equation}
	For (V), the right-hand side of equation (\ref{33}) can write
	\begin{equation}
	\begin{aligned}
    -x\frac{d}{dx}[\mathcal{X}[\mathbf{f}](x)]
=&-x\frac{d}{dx}({\sum_{n=0}^\infty}\mathbf{f}_{n}x^{-n})\\
=&-x{\sum_{n=0}^\infty}\mathbf{f}_{n} (-n)x^{-n-1}\\
=&{\sum_{n=0}^\infty}(n\mathbf{f}_{n})x^{-n}\\
=& \mathcal{X}[\mathbf{g}](x).
    \end{aligned}
	\end{equation}
	where $\quad |x|>\sigma_f$.
\end{proof}

\begin{theorem}
	Let ${x}\in \mathbb{C}$. If $\mathbf{f}$ and $\mathbf{g}$ are two biquaternion-valued number sequence, $\mathbf{w}$, called the convolution of $\mathbf{f}$ and $g$, by writing
	\begin{equation}\label{convolution}
	\begin{aligned} \mathbf{w}_n=&(\mathbf{f}*\mathbf{g})_n\\
=&{\sum_{k=0}^n}\mathbf{f}_{n-k}\mathbf{g}_k\\
=&{\sum_{k=0}^n}\mathbf{f}_{k}\mathbf{g}_{n-k},  \ n=0,1,2,\cdots.
    \end{aligned}
	\end{equation}
	The z transform $\mathcal{X}[\mathbf{f}](x)$ is convergence in $|x|>\sigma_f$ and $\mathcal{X}[g](x)$ is convergence in $|x|>\sigma_g$, where $x$ is complex valued. Then the convolution of $\mathbf{f}$ and $g$ has Z transform in $|x|>\max(\sigma_f,\sigma_g)$ and
	\begin{equation}
	\mathcal{X}[\mathbf{w}](x)=\mathcal{X}[\mathbf{f}](x)\mathcal{X}[\mathbf{g}](x).
	\end{equation}
\end{theorem}
\begin{proof}
	From the definition, we have
	\begin{equation}
    \begin{aligned}
	\mathcal{X}[\mathbf{w}](x)=&{\sum_{n=0}^\infty}\left({\sum_{k=0}^n}\mathbf{f}_{n-k}\mathbf{g}_k\right)x^{-n}\\
	=&{\sum_{k=0}^\infty}{\sum_{n=k}^\infty}\mathbf{f}_{n-k}\mathbf{g}_kx^{-n}.
	\end{aligned}
    \end{equation}
	Let $m=n-k$, we can rewrite
	\begin{equation}
    \begin{aligned}
	\mathcal{X}[\mathbf{w}](x)=&
	{\sum_{k=0}^\infty}{\sum_{n=k}^\infty}\mathbf{f}_{n-k}\mathbf{g}_kx^{-(n-k)}x^{-k}\\
	=&{\sum_{k=0}^\infty}\left({\sum_{m=0}^\infty}\mathbf{f}_{m}x^{-m}\right)\mathbf{g}_kx^{-k}.
    \end{aligned}
	\end{equation}
	Because $\mathbf{f}_{m}$ and ${g}_k$ both have z transform, then we get
	\begin{equation}
	\mathcal{X}[\mathbf{w}](x)
	={\sum_{m=0}^\infty}\mathbf{f}_{m}x^{-m}{\sum_{k=0}^\infty}\mathbf{g}_kx^{-k}
	=\mathcal{X}[\mathbf{f}](x)\mathcal{X}[\mathbf{g}](x).
	\end{equation}
	It is clear that since both sequences needs to converge in this domain, the final ROC is $|x|>\max(\sigma_f,\sigma_g)$.
\end{proof}
\begin{table*}
\caption{Properties of z transform}\label{table:p1}
\centering
\begin{tabular}{cccc}
		\toprule
		Property&Function&Z transform& ROC\\
		\midrule
		Left linearity &$\mathbf{c}_1\mathbf{f}+\mathbf{c}_2\mathbf{g}$ &$\mathbf{c}_1\mathcal{X}[\mathbf{f}](\mathbf{x})+\mathbf{c}_2\mathcal{X}[\mathbf{g}](\mathbf{x})$ & $\|\mathbf{x}\|>\max(\sigma_f,\sigma_g)$ \\
		Right linearity &$\mathbf{f}\mathbf{c}_1+\mathbf{g}\mathbf{c}_2$
		&$\mathcal{X}[\mathbf{f}](\mathbf{x})\mathbf{c}_1+\mathcal{X}[\mathbf{g}](\mathbf{x})\mathbf{c}_2 $&$\|\mathbf{x}\|>\max(\sigma_f,\sigma_g)$\\
		Two-side linearity &$\mathbf{c}_1\mathbf{f}+\mathbf{g}\mathbf{c}_2$&
		$\mathbf{c}_1\mathcal{X}[\mathbf{f}](\mathbf{x})+\mathcal{X}[\mathbf{g}](\mathbf{x})\mathbf{c}_2$&$\|\mathbf{x}\|>\max(\sigma_f,\sigma_g)$\\
		$\mathbf{q}^n$-scaling & $\mathbf{f}_n\mathbf{q}^n$ & $\mathcal{X}[\mathbf{f}](\mathbf{q}^{-1}\mathbf{x})$&$\|\mathbf{x}\|>{\sigma_f}{\|\mathbf{q}\|}$\\
		n-scaling &  $n\mathbf{f}_{n}$&  $-x\mathcal{X}[\mathbf{f}]'(x)$&$|x|>\sigma_f,~{x}\in \mathbb{C}$\\
		Shifting I & $\mathbf{f}_{n+k}$ & $\mathcal{X}[\mathbf{f}](\mathbf{x})\mathbf{x}^{k}-{\sum_{n=0}^{k-1}}\mathbf{f}_n\mathbf{x}^{k-n}$&$\|\mathbf{x}\|>\sigma_f$\\
		Shifting II & $\mathbf{f}_{n-k}$ & $\mathcal{X}[\mathbf{f}](\mathbf{x})\mathbf{x}^{-k}$&$\|\mathbf{x}\|>\sigma_f$\\
		Convolution & ${\sum_{k=0}^n}\mathbf{f}_{n-k}g_k$ &$\mathcal{X}[\mathbf{f}](x)\mathcal{X}[g](x)$&$|x|>\max(\sigma_f,\sigma_g),~{x}\in \mathbb{C}$\\
		\bottomrule
	\end{tabular}
\end{table*}

The properties of z transform are summarized in Table \ref{table:p1}. In the following, we give the z transformation of several special functions. 

\begin{example}
	\begin{equation}
	\begin{aligned}
	\mathcal{X}[1](\mathbf{x})=&(1-\mathbf{x}^{-1})^{-1}, \quad \|\mathbf{x}\|>1, \mathbf{x}\in \mathbb{H(C)},\\
	\mathcal{X}[n](\mathbf{x})=&\mathbf{x}(\mathbf{x}-1)^{-2}, \quad \|\mathbf{x}\|>1, \mathbf{x}\in \mathbb{H(C)},\\
	\mathcal{X}[n^2](\mathbf{x})=&(\mathbf{x}^2+\mathbf{x})(\mathbf{x}-1)^{-3}, \quad \|\mathbf{x}\|>1, \mathbf{x}\in \mathbb{H(C)},\\
	\mathcal{X}[\mathbf{p}^n]({x})=&(1-\mathbf{p}{x}^{-1})^{-1},\quad |{x}|>\|\mathbf{p}\|,
	{x}\in \mathbb{C}, \mathbf{p}\in \mathbb{H(C)},\\
	\mathcal{X}[n\mathbf{p}^n](x)=&\mathbf{p}(1-\mathbf{p}x^{-1})^{-1},\quad |x|>\|\mathbf{p}\|, {x}\in \mathbb{C},  \mathbf{p}\in \mathbb{H(C)}.
	\end{aligned}
	\end{equation}
\end{example}

\begin{example}
	Let $\mathbf{q} \in \mathbb{H(C)}$, and $\mathbf{q}=q_0+\underline{\mathbf{q}}=q_0+q_1\mathbf{i}+q_2\mathbf{j}+q_3\mathbf{k}$.
	According to \cite{cai}, we know
	\begin{equation}\label{cos}
	\cos (\mathbf{q}n)=
	\begin{cases}
	\frac{1}{2}\left(e^{{\frac{\underline{\mathbf{q}}}{|\underline{\mathbf{q}}|}\mathbf{q}n}}+
	e^{-{\frac{\underline{\mathbf{q}}}{|\underline{\mathbf{q}}|}\mathbf{q}n}}\right),& |\underline{\mathbf{q}}|\neq0\\
	\cos({q_0n})-\underline{\mathbf{q}}n\sin({q_0n}),&|\underline{\mathbf{q}}|=0,
	\end{cases}
	\end{equation}
	\begin{equation}\label{sin}
	\sin (\mathbf{q}n)=
	\begin{cases}
	-\frac{1}{2}\frac{\underline{\mathbf{q}}}{|\underline{\mathbf{q}}|}\left(e^{{\frac{\underline{\mathbf{q}}}{|\underline{\mathbf{q}}|}\mathbf{q}n}}-
	e^{-{\frac{\underline{\mathbf{q}}}{|\underline{\mathbf{q}}|}\mathbf{q}n}}\right),& |\underline{\mathbf{q}}|\neq0\\
	\sin({q_0n})+\underline{\mathbf{q}}n\cos({q_0n}),&|\underline{\mathbf{q}}|=0,
	\end{cases}
	\end{equation}
	and
	\begin{equation}\label{e}
	e^{\mathbf{q}}=
	\begin{cases}
	e^{{q}_0}(\cos(|\underline{\mathbf{q}}|)+sgn(\underline{\mathbf{q}})\sin(|\underline{\mathbf{q}}|)),& |\underline{\mathbf{q}}|\neq0\\
	e^{{q}_0}(1+\underline{\mathbf{q}}),&|\underline{\mathbf{q}}|=0.
	\end{cases}
	\end{equation}
	Then, we have
	\begin{equation}
    \begin{aligned}
	&\mathcal{X}[\cos (\mathbf{q}n)](\mathbf{x})\\
=&
	\begin{cases}
	\frac{1}{2}(1-e^{\frac{\underline{\mathbf{q}}}{|\underline{\mathbf{q}}|}\mathbf{q}}\mathbf{x}^{-1})^{-1}\\
	+\frac{1}{2}(1-e^{-\frac{\underline{\mathbf{q}}}{|\underline{\mathbf{q}}|}\mathbf{q}}\mathbf{x}^{-1})^{-1},& |\underline{\mathbf{q}}|\neq0\\
	\frac{\mathbf{x}^2-\mathbf{x}\cos(\mathbf{q})}{\mathbf{x}^2-2\mathbf{x}\cos(\mathbf{q})+1}, &|\underline{\mathbf{q}}|=0.
	\end{cases}
    \end{aligned}
	\end{equation}
	
	\begin{equation}
    \begin{aligned}
	&\mathcal{X}[\sin (\mathbf{q}n)](\mathbf{x})\\
=&
	\begin{cases}
	-\frac{1}{2}\frac{\underline{\mathbf{q}}}{|\underline{\mathbf{q}}|}(1-e^{\frac{\underline{\mathbf{q}}}{|\underline{\mathbf{q}}|}\mathbf{q}}\mathbf{x}^{-1})^{-1}\\
+\frac{1}{2}\frac{\underline{\mathbf{q}}}{|\underline{\mathbf{q}}|}(1-e^{-\frac{\underline{\mathbf{q}}}{|\underline{\mathbf{q}}|}\mathbf{q}}\mathbf{x}^{-1})^{-1},& |\underline{\mathbf{q}}|\neq0\\
	\frac{\mathbf{x}\sin(\mathbf{q})}{\mathbf{x}^2-2\mathbf{x}\cos(\mathbf{q})+1}, &|\underline{\mathbf{q}}|=0.
	\end{cases}
    \end{aligned}
	\end{equation}
\end{example}
\begin{proof}
	Without loss of generality,  we only give the proof of the $\mathcal{X}[\cos \mathbf{q}n](\mathbf{x})$ here. Now we see that, when $|\underline{\mathbf{q}}|=0$, $\mathcal{X}[\cos \mathbf{q}n](\mathbf{x})$ is actually doing the classical Z transformation. As for $|\underline{\mathbf{q}}|\neq0$, taking Z transform on both sides of the Eq. (\ref{cos}), we have
	\begin{equation}
	\begin{aligned}
	\mathcal{X}[\cos (\mathbf{q}n)](\mathbf{x})=&\sum_{m=0}^\infty
	\frac{1}{2}\left(e^{{\frac{\underline{\mathbf{q}}}{|\underline{\mathbf{q}}|}\mathbf{q}n}}+
	e^{-{\frac{\underline{\mathbf{q}}}{|\underline{\mathbf{q}}|}\mathbf{q}n}}\right)\mathbf{x}^{-n}\\
	=& \frac{1}{2}\sum_{m=0}^\infty
	e^{{\frac{\underline{\mathbf{q}}}{|\underline{\mathbf{q}}|}\mathbf{q}n}}\mathbf{x}^{-n}\\
    &+\frac{1}{2}\sum_{m=0}^\infty
	e^{-{\frac{\underline{\mathbf{q}}}{|\underline{\mathbf{q}}|}\mathbf{q}n}}\mathbf{x}^{-n},
	\end{aligned}
	\end{equation}
	and since $\mathcal{X}[\mathbf{p}^n](\mathbf{x})=(1-\mathbf{p} \mathbf{x}^{-1})^{-1}$, let $\mathbf{p}=e^{{\frac{\underline{\mathbf{q}}}{|\underline{\mathbf{q}}|}\mathbf{q}}}$, therefore,
	\begin{equation}
    \begin{aligned}
	\mathcal{X}[\cos (\mathbf{q}n)](\mathbf{x})=
	\frac{1}{2}(1-e^{{\frac{\underline{\mathbf{q}}}{|\underline{\mathbf{q}}|}\mathbf{q}}} \mathbf{x}^{-1})^{-1}\\
	+\frac{1}{2}(1-e^{{-\frac{\underline{\mathbf{q}}}{|\underline{\mathbf{q}}|}\mathbf{q}}} \mathbf{x}^{-1})^{-1},
	\quad \|\mathbf{x}\|>\|\mathbf{p}\|.
    \end{aligned}
	\end{equation}
\end{proof}

\begin{example}
	Let  $\mathbf{a}_n=C^m_{n+m}\mathbf{q}^n$, $\mathbf{b}_n=C^m_{n}\mathbf{q}^n$, and $\mathbf{c}_n=\frac{\mathbf{q}^n}{n!}$. Z transform of the sequences $\mathbf{a}=\{\mathbf{a}_n\}_{n=0}^\infty$, $\mathbf{b}=\{\mathbf{b}_n\}_{n=0}^\infty$, $\mathbf{c}=\{\mathbf{c}_n\}_{n=0}^\infty$ are the following:
	\begin{equation}
	\begin{aligned}
	\mathcal{X}[\mathbf{a}](\mathbf{x})=&
	(1-\mathbf{x}^{-1}\mathbf{q})^{-m}(1-\mathbf{q}\mathbf{x}^{-1})^{-1},\quad \|\mathbf{x}\|>\|\mathbf{q}\|,\\
	\mathcal{X}[\mathbf{b}](\mathbf{x})=&
	(\mathbf{x}\mathbf{q}^{-1}-1)^{-m}(1-\mathbf{q} \mathbf{x}^{-1})^{-1},\quad \|\mathbf{x}\|>\|\mathbf{q}\|,\\
	\mathcal{X}[\mathbf{c}](\mathbf{x})=&e^{\mathbf{q}\mathbf{x}^{-1}},\quad \|\mathbf{x}\|>\|\mathbf{q}\|,\\
	\end{aligned}
	\end{equation}
\end{example}
\begin{proof}
	\begin{enumerate}[(I)]
		\item Let $\mathbf{S}_m=\sum_{n=0}^\infty C^m_{n+m}\mathbf{q}^n\mathbf{x}^{-n}$,
		then
		\begin{equation}
		\begin{aligned}
		(1-\mathbf{x}^{-1}\mathbf{q})\mathbf{S}_m=&
		(1-\mathbf{x}^{-1}\mathbf{q})\sum_{n=0}^\infty C^m_{n+m}\mathbf{q}^n\mathbf{x}^{-n}\\
		=&\sum_{n=0}^\infty C^m_{n+m}\mathbf{q}^n\mathbf{x}^{-n}\\
       &-\sum_{n=1}^\infty C^m_{n+m-1}\mathbf{q}^n\mathbf{x}^{-n},
		\end{aligned}
		\end{equation}
		since the first part of the aforementioned equation is equal to $1+\sum_{n=1}^\infty C^m_{n+m}\mathbf{q}^n\mathbf{x}^{-n}$, then we have
		\begin{equation}
		\begin{aligned}
		(1-\mathbf{x}^{-1}\mathbf{q})\mathbf{S}_m=&
		1+\sum_{n=1}^\infty (C^{m}_{n+m}-C^{m}_{n+m-1})\mathbf{q}^n\mathbf{x}^{-n}\\
		=&1+\sum_{n=1}^\infty C^{m-1}_{n+m-1}\mathbf{q}^n\mathbf{x}^{-n}\\
		=&\mathbf{S}_{m-1}.
		\end{aligned}
		\end{equation}
		Therefore,
		\begin{equation}
        \begin{aligned}
		\mathbf{S}_m=&(1-\mathbf{x}^{-1}\mathbf{q})^{-1}\mathbf{S}_{m-1}\\
		=&(1-\mathbf{x}^{-1}\mathbf{q})^{-2}\mathbf{S}_{m-2},
        \end{aligned}
		\end{equation}
		and $\mathbf{S}_0=(1-\mathbf{q}\mathbf{x}^{-1})^{-1}$,
		then
		\begin{equation}
		\begin{aligned}
&\mathcal{X}[C^m_{n+m}\mathbf{q}^n](\mathbf{x})\\
=&(1-\mathbf{x}^{-1}\mathbf{q})^{-m}(1-\mathbf{q}\mathbf{x}^{-1})^{-1},\\
&\|\mathbf{x}\|>\|\mathbf{q}\|.
        \end{aligned}
		\end{equation}
		\item Let $\mathbf{S}_m=\sum_{n=m}^\infty C^m_{n}\mathbf{q}^n\mathbf{x}^{-n}$,
		then
		\begin{equation}
        \begin{aligned}
		&(\mathbf{x}\mathbf{q}^{-1}-1)\mathbf{S}_m\\
        =&\sum_{n=m}^\infty C^m_{n}\mathbf{q}^{n-1}\mathbf{x}^{-n+1}-\sum_{n=m}^\infty
		C^m_{n}\mathbf{q}^n\mathbf{x}^{-n},
        \end{aligned}
		\end{equation}
		since the first part of the above formula is equal to \\$\mathbf{q}^{m-1}\mathbf{x}^{-m+1}+\sum_{n=m+1}^\infty C^m_{n}\mathbf{q}^{n-1}\mathbf{x}^{-n+1}$, then
		\begin{equation}
		\begin{aligned}
		(\mathbf{x}\mathbf{q}^{-1}-1)\mathbf{S}_m=&
		\mathbf{q}^{m-1}\mathbf{x}^{-m+1}\\
        &+\sum_{n=m}^\infty
		( C^m_{n+1}- C^m_{n})\mathbf{q}^n\mathbf{x}^{-n}\\
		=&\sum_{n=m-1}^\infty
		C^{m-1}_{n}\mathbf{q}^n\mathbf{x}^{-n}\\
        =&\mathbf{S}_{m-1}.
		\end{aligned}
		\end{equation}
		Then, $\mathbf{S}_m=(\mathbf{x}\mathbf{q}^{-1}-1)^{-1}\mathbf{S}_{m-1}
		=(\mathbf{x}\mathbf{q}^{-1}-1)^{-2}\mathbf{S}_{m-2}$, and $\mathbf{S}_0=(1-\mathbf{q} \mathbf{x}^{-1})^{-1}$. Therefore,
		\begin{equation}
        \begin{aligned}
		\mathcal{X}[C^m_{n}\mathbf{q}^n](\mathbf{x})=&
		(\mathbf{x}\mathbf{q}^{-1}-1)^{-m}(1-\mathbf{q} \mathbf{x}^{-1})^{-1},\\
        \|\mathbf{x}\|>&\|\mathbf{q}\|.
        \end{aligned}
		\end{equation}
		\item First, let $d_n=\frac{1}{n!}$, then
		\begin{equation}
		\mathcal{X}[d_n](\mathbf{x})=\sum_{n=0}^\infty \frac{1}{n!}\mathbf{x}^{-n}=e^{\mathbf{x}^{-1}},\quad \|\mathbf{x}\|>0.
		\end{equation}
		By rule (ii) of Theorem \ref{th1},
		\begin{equation}
		\mathcal{X}[\frac{\mathbf{q}^n}{n!}](\mathbf{x})=
		\mathcal{X}[\frac{1}{n!}](\mathbf{q}^{-1}\mathbf{x})
		=e^{\mathbf{q}\mathbf{x}^{-1}}.
		\end{equation}
	\end{enumerate}
\end{proof}
\begin{table*}[!t]
	\caption{Some sequence $\mathbf{g}_n$ and their biquaternion z transform $\mathcal{X}[\mathbf{g}](x)$}
	\label{table:p2}
	\centering
	\begin{tabular}{ccc}
		\toprule
		& $\mathbf{g}_n$ & $\mathcal{X}[\mathbf{g}]$\\
		\midrule
		1 & 1               & $(1-\mathbf{x}^{-1})^{-1}$  \\
		2 & $n$             & $\mathbf{x}(\mathbf{x}-1)^{-2}$  \\
		3 & $n^2$           & $(\mathbf{x}^2+\mathbf{x})(\mathbf{x}-1)^{-3}$  \\
		4 & $\mathbf{p}^n$  & $(1-\mathbf{p}{x}^{-1})^{-1}$ \\
		5 & $n\mathbf{p}^n$ & $\mathbf{p}(1-\mathbf{p}x^{-1})^{-1}$ \\
		6 & $\cos \mathbf{q}n$ & $\frac{1}{2}(1-e^{{\frac{\underline{\mathbf{q}}}{|\underline{\mathbf{q}}|}\mathbf{q}}} \mathbf{x}^{-1})^{-1}+
		\frac{1}{2}(1-e^{{-\frac{\underline{\mathbf{q}}}{|\underline{\mathbf{q}}|}\mathbf{q}}} \mathbf{x}^{-1})^{-1}$\\
		7 & $\sin \mathbf{q}n$ &
		$-\frac{1}{2}\frac{\underline{\mathbf{q}}}{|\underline{\mathbf{q}}|}(1-e^{\frac{\underline{\mathbf{q}}}{|\underline{\mathbf{q}}|}\mathbf{q}}\mathbf{x}^{-1})^{-1}+
		\frac{1}{2}\frac{\underline{\mathbf{q}}}{|\underline{\mathbf{q}}|}(1-e^{-\frac{\underline{\mathbf{q}}}{|\underline{\mathbf{q}}|}\mathbf{q}}\mathbf{x}^{-1})^{-1}$\\
		8 & $C^m_{n+m}\mathbf{q}^n$ & $(1-\mathbf{x}^{-1}\mathbf{q})^{-m}(1-\mathbf{q}\mathbf{x}^{-1})^{-1}$\\
		9 & $C^m_{n}\mathbf{q}^n$ & $(\mathbf{x}\mathbf{q}^{-1}-1)^{-m}(1-\mathbf{q} \mathbf{x}^{-1})^{-1}$\\
		10 & $\frac{\mathbf{q}^n}{n!}$ & $e^{\mathbf{q}\mathbf{x}^{-1}}$\\
		\bottomrule
	\end{tabular}
\end{table*}

\section{Examples}\label{E}
Suppose sequence $\mathbf{f}_{n}$ is unknown and some initial values of $\mathbf{f}_{n}$ are given. Consider a class of biquaternion recurrence relations as follows
\begin{equation}\label{difference}
\begin{aligned}
&{\sum_{m=0}^M}\mathbf{f}_{n+m}\mathbf{p}_m\\
=&{\sum_{k=0}^M}\mathbf{g}_{n+k}\mathbf{q}_k+
{\sum_{t=0}^T}\mathbf{h}_{n+t}\mathbf{u}_t+\cdots+
{\sum_{l=0}^L}\mathbf{y}_{n+l}\mathbf{d}_l,
\end{aligned}
\end{equation}
where sequence $\mathbf{f}_{n}$,  $\mathbf{g}_{n}$, $\mathbf{h}_{n}$,$\cdots$, $\mathbf{y}_{n}$, and coefficients $\mathbf{p}_n$, $\mathbf{q}_n$, $\mathbf{u}_n$, $\cdots$, $\mathbf{d}_{n}$ are all known. If $\mathbf{g}_{n}=\mathbf{h}_{n}=\cdots=\mathbf{y}_{n}=0$ or
$\mathbf{q}_n=\mathbf{u}_n=\cdots=\mathbf{d}_{n}=0$, then we call the equation (\ref{difference}) homogeneous, otherwise it becomes inhomogeneous. Next, we will give several examples of the right-side coefficients homogeneous linear biquaternion equation and the right-side coefficients inhomogeneous linear biquaternion equation, respectively, to further explain how to solve the equations with the biquaternion z transform.

\begin{example}
	If we know that $\mathbf{f}_0=1$, $\mathbf{f}_1=\mathbf{i}+\mathbf{j}$ and
	\begin{equation}\label{ex1}
	\mathbf{f}_{n+2}=\mathbf{f}_{n+1}(\mathbf{i}+\mathbf{j}-1)+\mathbf{f}_n(\mathbf{i}+\mathbf{j}), \ n=0,1,2,\cdots,
	\end{equation}
	find a formular for $\mathbf{f}_n$.
\end{example}
Let $\mathcal{X}[\mathbf{f}](x)={\sum_{n=0}^\infty}\mathbf{f}_nx^{-n}$. It is clear that z transformation on both sides of the aforementioned equation yields an algebraic equation of
$\mathcal{X}[\mathbf{f}](x)$, which is
\begin{equation}\label{ex1-1}
\begin{aligned}
{\sum_{n=0}^\infty}\mathbf{f}_{n+2}x^{-n}
=&{\sum_{n=0}^\infty}\mathbf{f}_{n+1}x^{-n}(\mathbf{i}+\mathbf{j}-1)\\
&+{\sum_{n=0}^\infty}\mathbf{f}_nx^{-n}(\mathbf{i}+\mathbf{j}).
\end{aligned}
\end{equation}
We notice that, firstly,
\begin{equation}
\begin{aligned}
{\sum_{n=0}^\infty}\mathbf{f}_{n+1}x^{-n}
=&\left({\sum_{k=1}^\infty}\mathbf{f}_{k}x^{-k}\right)x\\
=&\left({\sum_{k=0}^\infty}\mathbf{f}_{k}x^{-k}-\mathbf{f}_0\right)x\\
=&\left(\mathcal{X}[\mathbf{f}](x)-1\right)x,
\end{aligned}
\end{equation}
and
\begin{equation}
\begin{aligned}
{\sum_{n=0}^\infty}\mathbf{f}_{n+2}x^{-n}
=&\left({\sum_{k=2}^\infty}\mathbf{f}_{k}x^{-k}\right)x^{2}\\
=&\left({\sum_{k=0}^\infty}\mathbf{f}_{k}x^{-k}-\mathbf{f}_0-\mathbf{f}_1x^{-1}\right)x^{2}\\
=&\left(\mathcal{X}[\mathbf{f}](x)-1-(\mathbf{i}+\mathbf{j})x^{-1}\right)x^{2}.
\end{aligned}
\end{equation}
Thus the equation (\ref{ex1-1}) can be written as
\begin{equation}
\begin{aligned}
&\left(\mathcal{X}[\mathbf{f}](x)-1-(\mathbf{i}+\mathbf{j})x^{-1}\right)x^{2}\\
=&\left(\mathcal{X}[\mathbf{f}](x)-1\right)x(\mathbf{i}+\mathbf{j}-1)
-\mathcal{X}[\mathbf{f}](x)(\mathbf{i}+\mathbf{j}),
\end{aligned}
\end{equation}
from which $\mathcal{X}[\mathbf{a}](x)$ can be solved. After simplification we have
\begin{equation}
\mathcal{X}[\mathbf{f}](x)=(1-(\mathbf{i}+\mathbf{j})x^{-1})^{-1}.
\end{equation}
We found in the preceding Table \ref{table:p2} that this is the z transformation of the sequence
\begin{equation}
\mathbf{f}_n=(\mathbf{i}+\mathbf{j})^n,  \ n=0,1,2,\cdots.
\end{equation}
Next, we can check the result by returning to the  initial conditions of the problem: $\mathbf{f}_0=1$, $\mathbf{f}_1=\mathbf{i}+\mathbf{j}$ are all right; and if $\mathbf{f}_n=(\mathbf{i}+\mathbf{j})^n$ and $\mathbf{f}_{n+1}=(\mathbf{i}+\mathbf{j})^{n+1}$, then
\begin{equation}\label{ex2}
\begin{aligned}
&\mathbf{f}_{n+1}(\mathbf{i}+\mathbf{j}-1)+\mathbf{f}_n(\mathbf{i}+\mathbf{j})\\
=&(\mathbf{i}+\mathbf{j})^{n+1}(\mathbf{i}+\mathbf{j}-1)+(\mathbf{i}+\mathbf{j})^n(\mathbf{i}+\mathbf{j})\\
=&(\mathbf{i}+\mathbf{j})^{n+2},
\end{aligned}
\end{equation}
which is also right.

\begin{example}
	Determine the numbers $\mathbf{f}_n$, $n=0,1,2,\cdots$, if~ $\mathbf{f}_0=1$, $\mathbf{f}_1=I\mathbf{j}$, and for $n=0,1,2,\cdots$
	\begin{equation}\label{ex3-1}
	\mathbf{f}_{n+2}=2\mathbf{f}_{n}+\mathbf{f}_{n+1}(I\mathbf{j}).
	\end{equation}
\end{example}
Let $\mathbf{F}(x)={\sum_{n=0}^\infty}\mathbf{f}_{n}x^{-n}$. It is clear that z transformation on both sides of equation (\ref{ex3-1}) yields an algebraic equation of
$\mathbf{F}(x)$, which is
\begin{equation}
{\sum_{n=0}^\infty}\mathbf{f}_{n+2}x^{-n}
=2{\sum_{n=0}^\infty}\mathbf{f}_{n}x^{-n}
+{\sum_{n=0}^\infty}\mathbf{f}_{n+1}x^{-n}(I\mathbf{j}).
\end{equation} Using (\ref{th3})  form the Theorem \ref{th1}, we get
\begin{equation}
[\mathbf{F}(x)-\mathbf{f}_0-\mathbf{f}_1x^{-1}]x^2=2\mathbf{F}(x)+
[\mathbf{F}(x)-\mathbf{f}_0]x,
\end{equation}
where $\mathbf{f}_0=1$, $\mathbf{f}_1=I\mathbf{j}$ and after computation we have
\begin{equation}
\mathbf{F}(x)=[1-(I\mathbf{j})x^{-1}]^{-1}.
\end{equation}
Then we get the expression of $\mathbf{f}_n$, that is
\begin{equation}
\mathbf{f}_n=(I\mathbf{j})^{n}.
\end{equation}
When we check the result, we find that it satisfies the initial conditions and equation (\ref{ex3-1}).

\begin{example}
	Find the numbers $\mathbf{f}_n$, $n=0,1,2,\cdots$, such that $\mathbf{f}_0=1$, $\mathbf{f}_1=I\mathbf{i}+I$ and $\mathbf{f}_{n+2}=\mathbf{f}_{n+1}(-2I^{-1})+2\mathbf{f}_{n}$.
\end{example}
Let $\mathbf{F}(x)={\sum_{n=0}^\infty}\mathbf{f}_{n}x^{-n}$. It is clear that z transformation on both sides of the aforementioned equation yields an algebraic equation of
$\mathbf{F}(x)$, which is
\begin{equation}
{\sum_{n=0}^\infty}\mathbf{f}_{n+2}x^{-n}
={\sum_{n=0}^\infty}\mathbf{f}_{n+1}x^{-n}(-2I^{-1})
+2{\sum_{n=0}^\infty}\mathbf{f}_{n}x^{-n}.
\end{equation}
Using (\ref{th3}) from the Theorem \ref{th1}, we get
\begin{equation}
[\mathbf{F}(x)-\mathbf{f}_0-\mathbf{f}_1x^{-1}]x^2=
[\mathbf{F}(x)-\mathbf{f}_0]x(-2I^{-1})+2\mathbf{F}(x),
\end{equation}
where $\mathbf{f}_0=1$, $\mathbf{f}_1=I\mathbf{i}+I$. Then
\begin{equation}
\mathbf{F}(x)(x^2+2I^{-1}x-2)=x(x+I\mathbf{i}+I+2I^{-1}).
\end{equation}
That is
\begin{equation}
\mathbf{F}(x)=[1-(I\mathbf{i}+I)x^{-1}]^{-1}.
\end{equation}
Using the equation (\ref{definition}) from definition \ref{de}, the final expression can be rewritten as
\begin{equation}
\mathbf{f}_n=(I\mathbf{i}+I)^n.
\end{equation}
Obviously, it satisfies the initial conditions.

\begin{remark}
	The solving process of the right-side coefficient linear homogeneous biquaternion recurrence relations is as follows:
	\begin{itemize}
		\item []{\bf{Step} 1}\label{step1}. Apply the biquaternion z transform (\ref{definition}) on both sides of the given equation.
		\item []{\bf{Step} 2}.\label{step2} Use the basic properties in Theorem \ref{th1} to simplify the result of step 1.
		\item []{\bf{Step} 3}. \label{step3} Using quaternion algebra, the result of step 2 is reduced to a monadic equation about the z transformation of the sequence.
		\item []{\bf{Step} 4}. \label{step4} The biquaternion sequence is obtained by comparing Table \ref{table:p2} and step 3.
	\end{itemize}
\end{remark}

The three homogeneous biquaternion recurrence relations with right-side coefficients have been solved by the biquaternion z transformation method. In the following, let's focus on the inhomogeneous biquaternion recurrence relations with the right-side coefficients.

\begin{example}
	Find $\mathbf{f}_n$, $n=0,1,2,\cdots$, such that $\mathbf{f}_0=\mathbf{f}_1=0$ and $\mathbf{f}_{n+2}-2\mathbf{f}_{n+1}+\mathbf{f}_n=2+(2-2I\mathbf{k})(I\mathbf{k})^n-(I\mathbf{k}+1)^n$.
\end{example}
It is clear that if we taking z transforms for both sides of above equation, it will gives
\begin{equation}
\begin{aligned}
&{\sum_{n=0}^\infty}\mathbf{f}_{n+2}x^{-n}-
2{\sum_{n=0}^\infty}\mathbf{f}_{n+1}x^{-n}+
{\sum_{n=0}^\infty}\mathbf{f}_{n}x^{-n}\\=
&2x(x-1)^{-1}+2(1-I\mathbf{k})[1-(I\mathbf{k})x^{-1}]^{-1}\\
&-[1-(I\mathbf{k}+1)x^{-1}]^{-1}.
\end{aligned}
\end{equation}
Let $\mathbf{F}(x)={\sum_{n=0}^\infty}\mathbf{f}_{n}x^{-n}$. Using (\ref{th3}) from the Theorem \ref{th1}, we get
\begin{equation}
\begin{aligned}
&x^2\mathbf{F}(x)-2x\mathbf{F}(x)+\mathbf{F}(x)\\
=&2x(x-1)^{-1}+2(1-I\mathbf{k})[1-(I\mathbf{k})x^{-1}]^{-1}\\
&-[1-(I\mathbf{k}+1)x^{-1}]^{-1}.
\end{aligned}
\end{equation}
That is
\begin{equation}\label{F}
\begin{aligned}
\mathbf{F}(x)
=&2(1-I\mathbf{k})(x-1)^{-2}[1-(I\mathbf{k})x^{-1}]^{-1}\\
&2x(x-1)^{-3}-(x-1)^{-2}[1-(I\mathbf{k}+1)x^{-1}]^{-1}.
\end{aligned}
\end{equation}
Noting that
\begin{equation}
\begin{aligned}
(x-1)^{-1}&={\sum_{n=0}^\infty}x^{-n-1},\\
(x-1)^{-2}&={\sum_{n=0}^\infty}(n+1)x^{-n-2},\\
(x-1)^{-3}&={\sum_{n=0}^\infty}\frac{(n+1)(n+2)}{2}x^{-n-3},
\end{aligned}
\end{equation}
and
\begin{equation}
\begin{aligned}\label{jj}
[1-(I\mathbf{k})x^{-1}]^{-1}=&{\sum_{n=0}^\infty}(I\mathbf{k})^nx^{-n},\\
[1-(I\mathbf{k}+1)x^{-1}]^{-1}=&{\sum_{n=0}^\infty}(I\mathbf{k}+1)^nx^{-n}\\
=&{\sum_{n=0}^\infty}2^{n-1}(I\mathbf{k}+1)x^{-n},\\
\end{aligned}
\end{equation}
which provide that
\begin{equation}
\begin{aligned}
2x(x-1)^{-3}=&2x{\sum_{n=0}^\infty}\frac{(n+1)(n+2)}{2}x^{-n-3}\\
=&{\sum_{n=0}^\infty}(n^2-n)x^{-n},
\end{aligned}
\end{equation}
\begin{equation}
\begin{aligned}
&2(1-I\mathbf{k})(x-1)^{-2}[1-(I\mathbf{k})x^{-1}]^{-1}\\
=&2(1-I\mathbf{k}){\sum_{n=0}^\infty}(n+1)x^{-n-2}{\sum_{n=0}^\infty}(I\mathbf{k})^nx^{-n}\\
=&{\sum_{n=0}^\infty}\{2n-2-2[(I\mathbf{k})+(I\mathbf{k})^2+\cdots+(I\mathbf{k})^{n-1}]\}\\
=&\left\{
\begin{array}{ll}
{\sum_{n=0}^\infty}\{2n-2-(n-1)(1+I\mathbf{k})\}x^{-n}, \hbox{when n is odd;} \\
{\sum_{n=0}^\infty}\{2n-2-(n-2)(1+I\mathbf{k})-I\mathbf{k}\}x^{-n}, \hbox{when n is even,}
\end{array}
\right.
\end{aligned}
\end{equation}
and
\begin{equation}
\begin{aligned}
&(x-1)^{-2}[1-(I\mathbf{k}+1)x^{-1}]^{-1}\\
=&{\sum_{n=0}^\infty}(n+1)x^{-n-2}{\sum_{n=0}^\infty}2^{n-1}(I\mathbf{k}+1)x^{-n}\\
=&{\sum_{n=0}^\infty}\{n-1+(2^{n-1}-n)(I\mathbf{k}+1)\}x^{-n}.
\end{aligned}
\end{equation}
Thus, when $n$ is odd, the equation (\ref{F}) can be rewritten as
\begin{equation}\label{odd}
\begin{aligned}
\mathbf{F}(x)
=&{\sum_{n=0}^\infty}\{n^2-n+2n-2-(n-1)(1+I\mathbf{k})\\
&+n-1+(2^{n-1}-n)(I\mathbf{k}+1)\}x^{-n}\\
=&{\sum_{n=0}^\infty}\{n^2-(I\mathbf{k}+1)^n+I\mathbf{k}\}x^{-n},
\end{aligned}
\end{equation}
and when $n$ is even, the equation (\ref{F}) can be rewritten as
\begin{equation}\label{even}
\begin{aligned}
\mathbf{F}(x)
=&{\sum_{n=0}^\infty}\{n^2-n+2n-2-(n-2)(1+I\mathbf{k})\\
&-I\mathbf{k}+n-1+(2^{n-1}-n)(I\mathbf{k}+1)\}x^{-n}\\
=&{\sum_{n=0}^\infty}\{n^2-(I\mathbf{k}+1)^n+1\}x^{-n}.
\end{aligned}
\end{equation}
From the above equations (\ref{odd}) and (\ref{even}), the final expression can be rewritten as
\begin{equation}
f_n=n^2-(I\mathbf{k}+1)^n+(I\mathbf{k})^n.
\end{equation}

We can also check the result by returning to the statement of the question: $\mathbf{f}_0=0$ and $\mathbf{f}_1=0$ are all right; and if
$\mathbf{f}_n=n^2-(I\mathbf{k}+1)^n+(I\mathbf{k})^n$, $\mathbf{f}_{n+1}=(n+1)^2-(I\mathbf{k}+1)^{n+1}+(I\mathbf{k})^{n+1}$, and
$\mathbf{f}_{n+2}=(n+2)^2-(I\mathbf{k}+1)^{n+2}+(I\mathbf{k})^{n+2}$, then
\begin{equation}
\begin{aligned}
&\mathbf{f}_{n+2}-2\mathbf{f}_{n+1}+\mathbf{f}_n\\
=&(n+2)^2-(I\mathbf{k}+1)^{n+2}+(I\mathbf{k})^{n+2}\\
&-2(n+1)^2+2(I\mathbf{k}+1)^{n+1}+n^2\\
&-2(I\mathbf{k})^{n+1}-(I\mathbf{k}+1)^n+(I\mathbf{k})^n\\
=&2+(2-2I\mathbf{k})(I\mathbf{k})^n-(I\mathbf{k}+1)^n,
\end{aligned}
\end{equation}
which is also right.

\begin{remark}
	Different from solving the right-side coefficient linear inhomogeneous biquaternion recurrence relations, the inhomogeneous cases are often more complicated, which increases the difficulty of calculation of quaternion algebras. In this case, we can consider using the Taylor expansion in the third step, and finally find the sequence by referring to the definition (\ref{definition}) and properties of the z transformation of the biquaternion.
\end{remark}

\begin{example}
	Find $\mathbf{f}(t)$, $t=0,1,2,\cdots$, from the equation
	\begin{equation}\label{ex22}
	{\sum_{n=0}^t}(3\mathbf{j})^{n}\mathbf{f}{(t-n)}=(2\mathbf{i})^{n}, t=0,1,2,\cdots.
	\end{equation}
\end{example}
It is clear that the left hand side is the convolution of $x$ and the function $t\rightarrow (3j)^t$, so that taking z transform of both members gives
\begin{equation}
({1-3\mathbf{j}x^{-1}})^{-1}\mathcal{X}[\mathbf{f}](x)=({1-2\mathbf{i}x^{-1}})^{-1}.
\end{equation}
Use the result of Eq. (\ref{definition}). We get
\begin{equation}
\begin{aligned}
&\mathcal{X}[\mathbf{f}](x)\\
=&({1-3\mathbf{j}x^{-1}})({1-2\mathbf{i}x^{-1}})^{-1}\\
=&({1-2\mathbf{i}x^{-1}})^{-1}-3\mathbf{j}x^{-1}({1-2\mathbf{i}x^{-1}})^{-1},
\end{aligned}
\end{equation}
and using Eq. (\ref{definition}) and rule (iv) of Theorem \ref{th1}, we see that
\begin{equation}
\mathbf{f}(t)=(2\mathbf{i})^t-3\mathbf{j}(2\mathbf{i})^{t-1}, t=0,1,2,\cdots.
\end{equation}

\section{Conclusion}
\label{C}
The solution of a class of biquaternion recurrence relations has been solved in this paper. After introducing the z transformation of the  biquaternion sequence, we give its basic properties. Besides, we discussed the relationship between a biquaternion-valued number sequence and a complex-valued number sequence: the convolution theorem. Moreover, applying the biquaternion z transform to the class of biquaternion recurrence relations can indicate the capability and efficiency of the method.

\section*{References}

\end{document}